\documentclass[11pt,reqno]{amsart}

\usepackage[T1]{fontenc}
\usepackage{lmodern}
\usepackage{microtype}
\usepackage{needspace}
\usepackage[margin=1.03in]{geometry}
\usepackage{amsmath,amssymb,mathtools}
\usepackage{amsthm}
\usepackage{xcolor}
\usepackage[colorlinks=true,
  linkcolor=blue!50!black,
  citecolor=blue!50!black,
  urlcolor=blue!50!black]{hyperref}
\usepackage[capitalize,noabbrev]{cleveref}

\allowdisplaybreaks[1]
\numberwithin{equation}{section}

\theoremstyle{plain}
\newtheorem{theorem}{Theorem}[section]
\newtheorem{proposition}[theorem]{Proposition}
\newtheorem{lemma}[theorem]{Lemma}
\newtheorem{corollary}[theorem]{Corollary}

\theoremstyle{definition}

\theoremstyle{remark}

\DeclareMathOperator{\Spec}{Spec}
\DeclareMathOperator{\Jac}{Jac}
\DeclareMathOperator{\gr}{gr}
\DeclareMathOperator{\ord}{ord}
\DeclareMathOperator{\height}{ht}
\newcommand{\N}{\mathbb N}
\newcommand{\m}{\mathfrak m}
\newcommand{\q}{\mathfrak q}
\newcommand{\p}{\mathfrak p}

\title[Dimensions of a ring and its power series ring]
  {Dimensions of a ring and its formal power series ring}

\author{Viet-Hoang Tran}
\address{Department of Mathematics, National University of Singapore, Singapore 119076}
\email{hoang.tranviet@u.nus.edu}
\urladdr{https://vh-tran.github.io/}

\author{Thieu N. Vo}
\address{Department of Computer Science, University of Bath, United Kingdom}
\email{ntv22@bath.ac.uk}

\author{Tan M. Nguyen}
\address{Department of Mathematics, National University of Singapore, Singapore 119076}
\email{tanmn@nus.edu.sg}
\urladdr{https://tanmnguyen89.github.io/}

\subjclass[2020]{Primary 13F25; Secondary 13C15, 13A18}
\keywords{Krull dimension, formal power series ring, local domain,
  strong finite type}

\hypersetup{
  pdftitle={Dimensions of a ring and its formal power series ring},
  pdfauthor={Viet-Hoang Tran}
}

\begin{document}
\raggedbottom

\begin{abstract}
Understanding the relation between $\dim R$ and $\dim R[[x]]$ is a classical problem in commutative algebra. For a Noetherian ring $R$, one has $\dim R[[x]]=\dim R+1$, but the general case is considerably more delicate. In 1973, Arnold proved that finite power-series dimension requires the strong finite type (SFT) condition, whereas, in 2002, Coykendall constructed a one-dimensional SFT domain whose power series ring has infinite dimension. The question of Coykendall and Gilmer whether $\dim R[[x]]<\infty$ forces $\dim R[[x]]\le2\dim R+1$ was answered negatively by Kang and Park in 2009. In this paper, we prove that, as $R$ ranges over the  nonzero commutative rings with identity, the finite pairs $(\dim R,\dim R[[x]])$ are exactly $(0,1)$ and the pairs $(n,m)$ with $1\le n<m$.
\end{abstract}


\maketitle
\enlargethispage{3pt}

\section{Background, related results, and the main theorem}
\label{sec:introduction}

Throughout, rings are nonzero, commutative, and have an identity,
and dimension means Krull dimension. The notation $R[[x]]$ denotes
the ring of formal power series in one indeterminate over $R$.
For an ideal $I\subseteq R$, we write
\[
 I[[x]]=\left\{\sum_{j\ge0}a_jx^j:a_j\in I\right\}.
\]
This is the kernel of the ring homomorphism $R[[x]]\longrightarrow(R/I)[[x]]$; it need not
equal the extended ideal $IR[[x]]$. We write $\N=\{1,2,\ldots\}$
and $\N_0=\N\cup\{0\}$.

The dimension theory of formal power series differs substantially
from that of polynomial rings. For a finite-dimensional ring $R$,
Seidenberg's bounds \cite{Seidenberg1953} give
\[
 \dim R+1\le\dim R[x]\le2\dim R+1.
\]
The lower bound also holds for $R[[x]]$, and equality holds when $R$
is Noetherian. In general, however, $R[[x]]$ can have infinite
dimension, and even its finite dimension need not satisfy the
polynomial upper bound.

An ideal $I$ is of \emph{strong finite type} (SFT) if there exist a
finitely generated ideal $J\subseteq I$ and an integer $e\ge1$ such
that $a^e\in J$ for every $a\in I$. A ring is SFT if each of its
ideals is SFT. Arnold \cite{Arnold1973a} proved that
\begin{equation}\label{eq:arnold}
 \dim R[[x]]<\infty\quad\Longrightarrow\quad R\text{ is SFT};
\end{equation}
see also Roitman~\cite{Roitman2015}. Condo, Coykendall, and
Dobbs~\cite[Corollary~2.2]{CondoCoykendallDobbs1996} proved that,
for a zero-dimensional SFT ring $R$,
\begin{equation}\label{eq:zero-classical}
 \dim R[[x_1,\ldots,x_s]]=s\qquad(s\ge1).
\end{equation}
These results account for the exceptional finite pair $(0,1)$.

Several important classes have more restrictive dimension formulas.
For a finite-dimensional Pr\"ufer domain $D$, Arnold's theorem
gives $\dim D[[x]]=\dim D+1$ precisely when $D$ is SFT, and
infinite dimension otherwise~\cite[Theorem~3.8]{Arnold1973b}.
More generally, if $D$ is an SFT Pr\"ufer domain of positive finite
dimension $d$, then (see \cite[Theorem~3.6]{Arnold1982})
\[
 \dim D[[x_1,\ldots,x_s]]=sd+1.
\]
The SFT hypothesis alone does not ensure finite power-series
dimension: Coykendall~\cite{Coykendall2002} constructed a
one-dimensional SFT domain whose power series ring has infinite
dimension.

Arnold~\cite{Arnold1981} had already shown that a finite
power-series dimension need not equal the coefficient dimension
plus one. Kang and Park~\cite{KangPark2009} constructed finite-dimensional
counterexamples to the bound $\dim R[[x]]\le2\dim R+1$.
Their mixed-extension formula yields, for suitable SFT Pr\"ufer
domains $D$ of dimension $d$ and $R=D[z_1,\ldots,z_p]$,
\[
 (\dim R,\dim R[[x]])=(d+p,\,d(p+1)+1).
\]
For example, $d=4$ and $p=3$ give $(7,17)$.
Thus finite violations of the polynomial bound are established
prior results. The construction below instead keeps the coefficient
dimension equal to one and prescribes every finite power-series
dimension at least two.

Our method also relates to the study of subrings of power series
rings defined by bounds on the degrees of their coefficients; see
Kang and Toan~\cite{KangToan2015a,KangToan2015b}.
Here the coefficients are rational functions, and the filtration
controls both their denominators and their numerator degrees.
We prove the required division estimate and Noetherianity directly.

\begin{theorem}\label{thm:classification}
Let $n,m\in\N_0$. There exists a nonzero commutative ring $R$ with
identity such that
\[
 \dim R=n\qquad\text{and}\qquad\dim R[[x]]=m
\]
if and only if
\[
 (n,m)=(0,1)\qquad\text{or}\qquad 1\le n<m.
\]
\end{theorem}

The essential realization statement is stronger in dimension one.

\begin{theorem}\label{thm:local}
For every integer $m\ge2$, there exists a one-dimensional local
domain $T$ such that $\dim T[[x]]=m$.
\end{theorem}

For $m=2$ one may take $T=k[[c]]$, where $k$ is a field.
For $m=r+2\ge3$, Sections~\ref{sec:filtration}--\ref{sec:dimension}
construct a local domain $T_r$ and prove
\begin{equation}\label{eq:target}
 \dim T_r=1,\qquad \dim T_r[[x]]=r+2.
\end{equation}
Section~\ref{sec:classification} deduces
\cref{thm:classification}. Its higher-dimensional realizations use
finite direct products; the theorem is stated for rings without a
domain hypothesis.

\section{The coefficient filtration and the local domain}
\label{sec:filtration}

Fix $r\ge1$ and let $k=\overline{\mathbb Q}$.
Let $u_1,\ldots,u_r$ be algebraically independent indeterminates
over $k$. Write $U=(u_1,\ldots,u_r)$ and
$k[U]=k[u_1,\ldots,u_r]$, and let $K=k(U)$ be the fraction field
of $k[U]$. Since $k[U]$ is countable,
we may enumerate representatives of its associate classes of
irreducible polynomials as $q_1,q_2,\ldots$. All polynomial degrees
below are total degrees in $U$; we set $\deg0=-\infty$.
For $N\in\N_0$, define
\begin{equation}\label{eq:windows}
 D_N=\prod_{j=1}^{N}q_j^{\lfloor N/j\rfloor},\qquad
 W_N=\left\{\frac{p}{D_N}:p\in k[U],\
             \deg p\le\deg D_N+N\right\}.
\end{equation}
Thus $D_0=1$ and $W_0=k$.

\begin{lemma}\label{lem:filtration}
The spaces $W_N$ are finite-dimensional over $k$ and satisfy
\begin{equation}\label{eq:filtration}
 W_N\subseteq W_{N+1},\qquad W_NW_M\subseteq W_{N+M},\qquad
 \bigcup_{N\ge0}W_N=K.
\end{equation}
Consequently the function $w(a)=\min\{N:a\in W_N\}$ on $K$
has $w(0)=0$ and satisfies
\begin{equation}\label{eq:weight}
 w(a+b)\le\max\{w(a),w(b)\},\qquad
 w(ab)\le w(a)+w(b).
\end{equation}
\end{lemma}

\begin{proof}
Finite dimensionality follows from the numerator degree bound.
The inclusions follow from
\[
 D_N\mid D_{N+1},\qquad D_ND_M\mid D_{N+M},
\]
the second divisibility being a consequence of
$\lfloor N/j\rfloor+\lfloor M/j\rfloor\le\lfloor(N+M)/j\rfloor$.
For example, after writing a product over $D_{N+M}$, its numerator
has degree at most
\[
 (\deg D_N+N)+(\deg D_M+M)
       +\deg\frac{D_{N+M}}{D_ND_M}
 =\deg D_{N+M}+N+M.
\]
For $p/q\in K$, choose $N$ such that $q\mid D_N$ and
$N\ge\deg p-\deg q$. Then $p/q\in W_N$, proving exhaustion.
The inequalities for $w$ follow from the vector-space and
multiplicative properties of the filtration.
\end{proof}

Set $V_j=W_{2^j-1}$ and define a subring of $K[[c]]$ by
\begin{equation}\label{eq:T}
 T=T_r=\left\{\sum_{j\ge0}a_jc^j:a_j\in V_j\right\}.
\end{equation}
Indeed, $V_iV_j\subseteq V_{i+j}$ because
$(2^i-1)+(2^j-1)\le2^{i+j}-1$. In particular $T$ is a domain,
$c\in T$, and every constant coefficient belongs to $k$.

We will use the elementary estimate
\begin{equation}\label{eq:composition}
 \sum_{i=1}^{s}2^{j_i}\le2^n
 \quad\text{if }j_i\ge1\text{ and }\sum_{i=1}^{s}j_i=n.
\end{equation}
It follows by repeatedly applying $2^a+2^b\le2^{a+b}$ for
$a,b\ge1$. Hence, in a ring with a weight satisfying
\eqref{eq:weight} and $w(1)=w(-1)=0$, the condition
$w(b_j)\le C2^j$ for $j\ge1$ implies that the coefficients $d_n$ in
\begin{equation}\label{eq:reciprocal}
 \left(1+\sum_{j\ge1}b_jc^j\right)^{-1}=\sum_{n\ge0}d_nc^n
\end{equation}
satisfy $w(d_n)\le C2^n$. Indeed, $d_0=1$, and for $n\ge1$,
\[
 d_n=\sum_{s=1}^{n}(-1)^s
       \sum_{\substack{j_1+\cdots+j_s=n\\j_i\ge1}}
           b_{j_1}\cdots b_{j_s}.
\]

\begin{proposition}\label{prop:local}
The ring $T$ is local, with maximal ideal
\[
 \m=\left\{\sum_{j\ge0}a_jc^j\in T:a_0=0\right\}.
\]
Moreover,
\begin{equation}\label{eq:local-properties}
 \m^3\subseteq cT,\qquad \sqrt{cT}=\m,\qquad
 T[1/c]\text{ is a field},\qquad \Spec T=\{(0),\m\}.
\end{equation}
In particular, $\dim T=1$.
\end{proposition}

\begin{proof}
If $a\in T$ has nonzero constant coefficient, scale that
coefficient to one. The expansion in~\eqref{eq:reciprocal} and
$V_{j_1}\cdots V_{j_s}\subseteq V_{j_1+\cdots+j_s}$ show that
its formal inverse belongs to $T$. Thus $T$ is local and $T/\m=k$.

For positive $i,j,\ell$ with $i+j+\ell=n$, one has
\begin{equation}\label{eq:cube-estimate}
 (2^i-1)+(2^j-1)+(2^\ell-1)
 \le2^{n-2}+1\le2^{n-1}-1\qquad(n\ge3).
\end{equation}
For the first inequality, shifting one unit from a smaller index
greater than one to a largest index does not decrease the sum of
the powers of two; the maximum is attained at $(n-2,1,1)$.
Every coefficient of $c^n$ in a product of three elements of $\m$
therefore belongs to $V_{n-1}$. Division by $c$ gives
$\m^3\subseteq cT$, and $\sqrt{cT}=\m$ follows.

If $w(b_j)\le C2^j$ for all $j\ge0$, then
\begin{equation}\label{eq:shift}
 c^d\sum_{j\ge0}b_jc^j\in T
 \quad\text{whenever }2^d\ge C+1,
\end{equation}
since $C2^j\le2^{j+d}-1$. For $0\ne a\in T$, write
\[
 a=c^\nu a_\nu\left(1+\sum_{j\ge1}b_jc^j\right),
 \qquad b_j=a_{\nu+j}a_\nu^{-1},\quad a_\nu\ne0.
\]
The estimate
$w(b_j)\le2^{\nu+j}-1+w(a_\nu^{-1})$ gives an exponential
bound. Formula~\eqref{eq:reciprocal} gives the same type of bound
for the reciprocal of the parenthesized series; multiplication by
the fixed scalar $a_\nu^{-1}$ preserves it.
Equation~\eqref{eq:shift}, with $c^{-\nu}$ absorbed into the
shift, yields $c^d/a\in T$ for some $d$.
Thus every nonzero element of $T$ is invertible in $T[1/c]$.

Primes avoiding $c$ correspond to primes of this field, so the only
such prime is $(0)$. The radical identity shows that $\m$ is the
only prime containing $c$. Since $0\ne c\in\m$, the spectrum
and dimension assertions follow.
\end{proof}

\section{Bounded division and an auxiliary Noetherian ring}
\label{sec:division}

Put
\begin{equation}\label{eq:B}
 F=k[[x]],\quad L=k((x)),\quad A_0=F[U],\quad
 S=k[U]\setminus\{0\},\quad B=S^{-1}A_0.
\end{equation}
The natural embeddings identify $B$ with $K\otimes_kF$ inside
$K[[x]]$. The ring $B$ is a Noetherian domain, and
\begin{equation}\label{eq:Bdim}
 \dim B\le r.
\end{equation}
Indeed, if $\q\in\Spec B$ avoids $x$, its lower prime chains lie
in $B[1/x]$, a localization of $L[U]$, and hence $\height\q\le r$.
If $x\in\q$, then $B/xB=K$ gives $\q=xB$. The principal ideal
theorem gives $\height(xB)=1\le r$.

Extend the filtration to $B$ by
\begin{equation}\label{eq:E}
 E_N=W_N\otimes_kF
 =\left\{\frac{p}{D_N}:p\in F[U],\
                    \deg p\le\deg D_N+N\right\}.
\end{equation}
Under the embedding into $K[[x]]$, one also has
\begin{equation}\label{eq:finite-basis}
 W_N[[x]]=E_N.
\end{equation}
Here the left side means series whose coefficients lie in $W_N$.
To verify the identity, expand each coefficient in the finite
$k$-basis $\{U^\alpha/D_N:|\alpha|\le\deg D_N+N\}$ and collect
the coefficients of each basis element into a series in $F$.
The $E_N$ exhaust $B$, are increasing and multiplicative, and define
$w(b)=\min\{N:b\in E_N\}$ on $B$. The inequalities
\eqref{eq:weight} still hold, and every element of $F$ has weight zero.

\begin{lemma}[Bounded division]\label{lem:division}
Fix $b_1,\ldots,b_s\in B$ and let $J=(b_1,\ldots,b_s)B$.
There is an integer $C\ge0$ such that, for every $N\ge0$ and
$h\in J\cap E_N$, one can write
\begin{equation}\label{eq:division}
 h=\sum_{i=1}^{s}b_i z_i,\qquad z_i\in E_{N+C}.
\end{equation}
\end{lemma}

\begin{proof}
The case $J=0$ is immediate. Set $J_0=J\cap A_0$ and filter
$A_0=F[U]$ by total degree in $U$. Its associated graded ring is
$F[U]$, which is Noetherian. Thus the homogeneous ideal generated
by the highest-degree parts of elements of $J_0$ is generated by
the highest-degree parts of finitely many nonzero elements
$\ell_1,\ldots,\ell_t\in J_0$. Write $d_j=\deg\ell_j$.
Every $0\ne a\in J_0$ admits a representation
\begin{equation}\label{eq:degree-division}
 a=\sum_{j=1}^{t}\ell_jp_j,\qquad
 p_j\in F[U],\qquad \deg p_j\le\deg a-d_j.
\end{equation}
Indeed, express the highest-degree part of $a$ using homogeneous
multipliers of degree $\deg a-d_j$, subtract the corresponding
combination of the $\ell_j$, and repeat at the smaller degree.
This process terminates. A negative degree bound means $p_j=0$;
no division by a nonunit of $F$ is involved.

Since $\ell_j\in J$, there are fixed $s_0\in S$ and
$A_{ij}\in F[U]$ with
\[
 \ell_j=\sum_{i=1}^{s}b_i\frac{A_{ij}}{s_0}
 \qquad(1\le j\le t).
\]
For $h\in J\cap E_N$, write $h=a/D_N$, where
$a\in A_0$ and $\deg a\le\deg D_N+N$.
Since $D_N$ is a unit of $B$, we have $a\in J_0$.
Using~\eqref{eq:degree-division}, set
\[
 z_i=\frac{\sum_j A_{ij}p_j}{s_0D_N};
 \qquad h=\sum_i b_i z_i.
\]
There is a fixed $C_0\ge0$ such that each numerator here has degree
at most $\deg D_N+N+C_0$.

Choose $C$ so large that
\begin{equation}\label{eq:denominator-shift}
 s_0D_N\mid D_{N+C}\quad(N\ge0),\qquad C\ge C_0-\deg s_0.
\end{equation}
To obtain the divisibility, for each factor $q_j^{e_j}$ of $s_0$
it suffices to require $C\ge je_j$, since
\[
 \left\lfloor\frac{N+C}{j}\right\rfloor
 -\left\lfloor\frac{N}{j}\right\rfloor\ge e_j.
\]
Writing $z_i$ over $D_{N+C}$ now bounds its numerator degree by
\[
 \deg D_N+N+C_0+\deg\frac{D_{N+C}}{s_0D_N}
 =\deg D_{N+C}+N+C_0-\deg s_0
 \le\deg D_{N+C}+N+C.
\]
Thus $z_i\in E_{N+C}$, as required.
\end{proof}

Inside $B[[c]]$, define
\begin{equation}\label{eq:H}
 H=\left\{\sum_{n\ge0}b_nc^n:
       \text{there is }M\ge0\text{ with }w(b_n)\le M2^n
       \text{ for all }n\right\}.
\end{equation}
The bound may depend on the series. This is a subring containing
$B$: for multiplication, the weight of the coefficient of $c^n$
is at most
\[
 \max_{0\le i\le n}(M2^i+M'2^{n-i})\le(M+M')2^n.
\]
Shifting the coefficients gives
\begin{equation}\label{eq:Hgraded}
 H\cap c^eB[[c]]=c^eH\quad(e\ge0),\qquad
 H/cH\cong B,\qquad \gr_cH\cong B[z],
\end{equation}
where $z$ is the initial form of $c$. Furthermore,
\begin{equation}\label{eq:jacobson}
 c\in\Jac(H).
\end{equation}
Indeed, for every $h\in H$, the positive-degree coefficients of
$ch$ satisfy an exponential bound. Formula~\eqref{eq:reciprocal}
shows that $(1-ch)^{-1}\in H$.

\Needspace{6\baselineskip}
\begin{proposition}\label{prop:noetherian}
The ring $H$ is Noetherian, and $\dim H[1/c]\le r$.
\end{proposition}

\begin{proof}
Let $0\ne I\subseteq H$ be an ideal. Since $B[z]$ is Noetherian,
the ideal of initial forms of $I$ is generated by the initial forms
of finitely many nonzero elements $f_1,\ldots,f_s\in I$. Write
\[
 e_i=\ord_c(f_i),\qquad
 f_i=\sum_{j\ge0}f_{i,e_i+j}c^{e_i+j},\qquad b_i=f_{i,e_i}.
\]
For each $n\ge0$ set $J_n=(b_i:e_i\le n)B$.
There are only finitely many such generator lists, so
\cref{lem:division} supplies one constant $C\ge0$ valid for
all the nonzero $J_n$.

Fix $g=\sum_{n\ge0}g_nc^n\in I$. Choose $A,F_0\ge0$ with
\begin{equation}\label{eq:input-bounds}
 w(g_n)\le A2^n\quad(n\ge0),\qquad
 w(f_{i,e_i+j})\le F_0 2^j\quad(j\ge1,\ 1\le i\le s).
\end{equation}
We cancel the coefficients of $g$ successively. After canceling
the coefficients below degree $n$ by finitely many multiples of
the $f_i$, the residual still lies in $I$. Its coefficient $h_n$
at degree $n$ belongs to $J_n$: if nonzero, its initial form lies
in the homogeneous ideal generated by $b_i z^{e_i}$.
By bounded division choose $a_{i,n}\in B$ such that
\begin{equation}\label{eq:cancellation}
 h_n=\sum_{e_i\le n}b_i a_{i,n},\qquad
 w(a_{i,n})\le w(h_n)+C,
\end{equation}
and put $a_{i,n}=0$ when $e_i>n$.
If $h_n=0$, choose all coefficients to be zero.
Subtracting $\sum_{e_i\le n}a_{i,n}c^{n-e_i}f_i$ cancels
degree $n$ as well.

We must verify that the resulting infinite multipliers belong to
$H$. Put $T_n=\max_i w(a_{i,n})$. A cancellation at degree
$m<n$ contributes to degree $n$ a product
$a_{i,m}f_{i,e_i+n-m}$. Thus~\eqref{eq:weight},
\eqref{eq:input-bounds}, and~\eqref{eq:cancellation} give
\begin{equation}\label{eq:recurrence}
 T_n\le\max\left\{A2^n,
       \max_{0\le m<n}\bigl(T_m+F_0 2^{n-m}\bigr)\right\}+C,
\end{equation}
where the inner maximum is omitted for $n=0$.
Choose $M\ge\max\{A+C,\,2F_0+C\}$.
We claim $T_n\le M2^n$. The case $n=0$ is immediate.
If the claim holds below $n$, then, for $d=n-m\ge1$,
\begin{align*}
 M(2^n-2^m)
 &=M2^m(2^d-1)\\
 &\ge(2F_0+C)(2^d-1)\ge F_0 2^d+C.
\end{align*}
It follows that $T_m+F_0 2^{n-m}+C\le M2^n$; also
$A2^n+C\le M2^n$. Equation~\eqref{eq:recurrence} proves
the induction step. Consequently
\[
 h_i=\sum_{n\ge e_i}a_{i,n}c^{n-e_i}\in H,
\]
because its coefficient of degree $j$ has weight at most
$M2^{e_i}2^j$. The cancellation gives the coefficientwise identity
$g=\sum_i f_i h_i$ in $B[[c]]$, and hence in $H$.
Thus $I=(f_1,\ldots,f_s)H$. The zero ideal is finitely generated
as well, proving that $H$ is Noetherian.

For every maximal ideal $\mathfrak n$ of $H$,
\eqref{eq:jacobson} and the Noetherian local dimension inequality
for a principal quotient give
\[
 \dim H_{\mathfrak n}
 \le\dim(H_{\mathfrak n}/cH_{\mathfrak n})+1
 \le\dim B+1\le r+1.
\]
Therefore $\dim H\le r+1$. Every prime chain avoiding $c$ can
be extended strictly by a maximal ideal, since every maximal ideal
contains $c$. Thus every such chain has length at most $r$, and
\mbox{$\dim H[1/c]\le r$}.
\end{proof}

\section{The dimension of the formal power series ring}
\label{sec:dimension}

Let $A=T[[x]]$. Interchanging the $x$- and $c$-coefficients and
using~\eqref{eq:finite-basis} identifies $A$ with the subring
\begin{equation}\label{eq:array}
 A=\left\{\sum_{n\ge0}b_nc^n\in B[[c]]:
                         b_n\in E_{2^n-1}\text{ for all }n\right\}.
\end{equation}
Indeed, a series in $T[[x]]$ gives at each $c$-degree $n$ an
element of $V_n[[x]]=E_{2^n-1}$. Conversely, expanding each $b_n$
in $x$ gives coefficients in $V_n$, so each resulting
$x$-coefficient belongs to $T$. Multiplication agrees because
each coefficient of $x^i c^j$ involves only a finite sum.

\Needspace{5\baselineskip}
\begin{proposition}\label{prop:upper}
One has $\dim A\le r+2$.
\end{proposition}

\begin{proof}
Equation~\eqref{eq:array} gives $A\subseteq H$.
If $h=\sum b_nc^n\in H$ has bound $w(b_n)\le M2^n$, choose
$d$ with $2^d\ge M+1$. Then
$M2^n\le2^{n+d}-1$, so $c^dh\in A$. Consequently
\begin{equation}\label{eq:localization}
 A[1/c]=H[1/c],\qquad\dim A[1/c]\le r.
\end{equation}
Put $P=\m[[x]]$. Since $\m^3\subseteq cT$,
\[
 P^3\subseteq(cT)[[x]]=cA\subseteq P.
\]
The equality in the middle follows by dividing coefficients by
$c$. Since $A/P\cong k[[x]]$, we obtain
\begin{equation}\label{eq:closed-part}
 \sqrt{cA}=P,\qquad\dim(A/cA)=1.
\end{equation}
In a finite prime chain of $A$, the segment avoiding $c$ has
length at most $r$, and the segment containing $c$ has length at
most one. If both occur, there is one additional inclusion between
them. Thus every such chain has length at most $r+2$.
\end{proof}

\begin{lemma}\label{lem:exponentials}
There exist $f_1,\ldots,f_r\in xk[[x]]$ algebraically
independent over $k$.
\end{lemma}

\begin{proof}
Choose $\lambda_1,\ldots,\lambda_r\in k$ linearly independent
over $\mathbb Q$ and set
\[
 f_i(x)=\exp(\lambda_i x)-1
       =\sum_{n\ge1}\frac{\lambda_i^n}{n!}x^n.
\]
Such $\lambda_i$ exist: for $r>1$, a root $\alpha$ of the
Eisenstein polynomial $t^r-2$ gives $1,\alpha,\ldots,\alpha^{r-1}$;
for $r=1$ take $\lambda_1=1$.
A polynomial relation in the $f_i$, after an invertible polynomial
translation, would yield
\[
 \sum_{\nu}a_\nu\exp(\mu_\nu x)=0,
 \qquad \mu_\nu=\sum_i\nu_i\lambda_i,
\]
with finitely many distinct multi-indices $\nu$ and not all
$a_\nu\in k$ zero. The $\mu_\nu$ are distinct. List them as
$\mu_1,\ldots,\mu_s$. Formal differentiation at zero in orders
$0,\ldots,s-1$ gives a homogeneous linear system with determinant
\[
 \det(\mu_j^{i-1})_{1\le i,j\le s}
   =\prod_{1\le i<j\le s}(\mu_j-\mu_i)\ne0,
\]
a contradiction. All operations are formal in characteristic zero.
\end{proof}

\begin{proposition}\label{prop:lower}
The ring $A$ contains a prime chain of length $r+2$.
Thus $\dim T[[x]]=r+2$.
\end{proposition}

\begin{proof}
Choose the series in \cref{lem:exponentials} and put
\[
 \q_j=(u_1-f_1(x),\ldots,u_j-f_j(x))B\quad(1\le j\le r),
 \qquad\q_0=(0).
\]
Before localizing, substitution of $f_i(x)$ for $u_i$, $i\le j$,
defines a surjection
\[
 F[U]\longrightarrow F[u_{j+1},\ldots,u_r]
\]
with kernel $(u_1-f_1(x),\ldots,u_j-f_j(x))F[U]$.
Its kernel is disjoint from $S$: expanding a nonzero element of
$k[U]$ in the remaining variables would otherwise give a nonzero
polynomial relation over $k$ among $f_1,\ldots,f_j$.
Localization therefore gives proper prime ideals and a strict chain
\begin{equation}\label{eq:qchain}
 (0)=\q_0\subsetneq\q_1\subsetneq\cdots\subsetneq\q_r.
\end{equation}
For strictness, $u_j-f_j(x)$ is nonzero in the domain obtained
after the preceding substitution, and remains nonzero after
localization. Full substitution extends to $B\longrightarrow L$
with kernel $\q_r$, since every element of $S$ has nonzero image
in $F\subseteq L$. This map fixes $F$, so
\begin{equation}\label{eq:qF}
 \q_r\cap F=(0).
\end{equation}

The coefficientwise ideal $\q_j[[c]]\subseteq B[[c]]$ is prime,
since its quotient is $(B/\q_j)[[c]]$, a domain.
Contract these primes to $A$:
\[
 Q_j=\q_j[[c]]\cap A\qquad(0\le j\le r).
\]
They all avoid $c$, and $Q_0=(0)$.
Since $u_j=D_1u_j/D_1\in W_1$ and $f_j(x)\in F$, one has
\[
 c(u_j-f_j(x))\in Q_j\setminus Q_{j-1}
 \qquad(1\le j\le r),
\]
where membership in $A$ follows from~\eqref{eq:array}.
Thus the contracted chain is strict.

If $a\in Q_r$, its coefficient at $c^0$ belongs to $E_0=F$
by~\eqref{eq:array} and to $\q_r$ by definition. It is zero
by~\eqref{eq:qF}. This says exactly that every $x$-coefficient
of $a$ belongs to $\m$, so $Q_r\subseteq P=\m[[x]]$.
The inclusion is strict because $c\in P\setminus Q_r$.
Finally $P\subsetneq P+(x)$ is a strict prime inclusion, since
$A/P\cong k[[x]]$. Hence
\begin{equation}\label{eq:final-chain}
 (0)=Q_0\subsetneq Q_1\subsetneq\cdots\subsetneq Q_r
       \subsetneq P\subsetneq P+(x)
\end{equation}
has length $r+2$. Combining this with \cref{prop:upper} proves
the dimension equality.
\end{proof}

\begin{proof}[Proof of \cref{thm:local}]
For $m=2$ use the discrete valuation ring $k[[c]]$.
For $m\ge3$ set $r=m-2$ and use $T_r$ from~\eqref{eq:T}.
Its required dimensions follow from \cref{prop:local,prop:lower}.
\end{proof}

\section{The classification of finite pairs}\label{sec:classification}

\begin{proof}[Proof of \cref{thm:classification}]
Suppose first that $\dim R=n$ and $\dim R[[x]]=m$ are finite.
Choose a chain of prime ideals
\[
  \p_0\subsetneq\p_1\subsetneq\cdots\subsetneq\p_n
\]
in $R$.  For each $i$, coefficientwise reduction gives
\[
  R[[x]]/\p_i[[x]]\cong (R/\p_i)[[x]],
\]
which is a domain.  Thus each coefficientwise ideal $\p_i[[x]]$
is prime.  Their inclusions are strict, as witnessed by constant
series from $\p_i\setminus\p_{i-1}$.  Moreover,
\[
  R[[x]]/(\p_n[[x]]+(x))\cong R/\p_n
\]
is a domain, and $x\notin\p_n[[x]]$.  Consequently
\[
  \p_0[[x]]\subsetneq\cdots\subsetneq\p_n[[x]]
  \subsetneq\p_n[[x]]+(x)
\]
is a prime chain of length $n+1$, proving
\begin{equation}\label{eq:necessary-lower-bound}
  m\ge n+1.
\end{equation}
If $n=0$, finiteness of $m$ and \cref{eq:arnold} imply that $R$
is an SFT ring.  The zero-dimensional result
\cref{eq:zero-classical} then gives $m=1$.  Hence every finite
pair has the asserted form.

Conversely, a field realizes $(0,1)$.  Fix integers
$1\le n<m$.  By \cref{thm:local}, choose a local integral domain
$T$ such that
\[
  \dim T=1,\qquad \dim T[[x]]=m.
\]
Let
\[
  C=k[z_1,\ldots,z_n],\qquad R=T\times C.
\]
The polynomial ring $C$ is Noetherian of dimension $n$, so the
Noetherian power-series dimension formula yields
$\dim C[[x]]=n+1$.  The prime spectrum of a finite direct
product is the disjoint union of the spectra of its factors,
and coefficientwise projection gives
\[
  (T\times C)[[x]]\cong T[[x]]\times C[[x]].
\]
It follows that
\[
  \dim R=\max\{1,n\}=n,\qquad
  \dim R[[x]]=\max\{m,n+1\}=m.
\]
This realizes every asserted pair.
\end{proof}

\Needspace{10\baselineskip}
\begin{corollary}\label{cor:no-finite-upper-bound}
For every fixed integer $n\ge1$, the set
\[
  \bigl\{\dim R[[x]]:\dim R=n,\ \dim R[[x]]<\infty\bigr\}
\]
is unbounded.  For $n=1$, it remains unbounded when $R$ is
required to be a local integral domain.
\end{corollary}

\begin{proof}
The first assertion follows from \cref{thm:classification},
and the second follows from \cref{thm:local}.
\end{proof}

\renewcommand{\sectionname}{}
\section*{Acknowledgments}
We used GPT-5.6 Sol and an agentic harness built around GPT-5.6 Sol and Claude Fable 5 to assist with literature searches, hypothesis testing, the exploration and elimination of potential approaches, wording refinement, and manuscript proofreading. We thank Hieu M. Vu, Tho Tran Huu, Khoi M. N. Nguyen, Dung V. Nguyen, and Quang X. Nguyen for their assistance with hardware-related matters and for providing technical support in the use of the AI tools and agentic system. 
\bibliographystyle{plain}
\bibliography{references}

\end{document}